\documentclass[12pt]{amsart}
\usepackage{extsizes}
\usepackage{amssymb, amsmath}

\usepackage{geometry}
\theoremstyle{plain}
\newtheorem{theorem}{Theorem}
\newtheorem{corollary}[theorem]{Corollary}
\newtheorem{proposition}[theorem]{Proposition}
\newtheorem{lemma}[theorem]{Lemma}

\theoremstyle{remark}

\newcommand{\Bh}{B(H)}
\newcommand{\Bs}{B(H)_s}
\newcommand{\Bp}{B(H)^+}
\newcommand{\Bpp}{B(H)^{++}}
\newcommand{\R}{\mathbb R}
\renewcommand{\l}{\lambda}
\newcommand{\Mle}{M_{le}}

\newcommand{\ler}[1]{\left( #1 \right)}

\newcommand{\A}{\mathcal A}

\newcommand{\App}{\mathcal A^{++}}

\newcommand{\fel}{1/2}
\newcommand{\mfel}{{-1/2}}

\newcommand{\PI}{P_1(H)}

\begin{document}

\title[]{Characterizations of certain means of positive operators}

\author{Lajos Moln\'ar}
\address{Bolyai Institute, Interdisciplinary Excellence Centre, University of Szeged,
H-6720 Szeged, Aradi v\'ertan\'uk tere 1.,
Hungary, and Institute of Mathematics
Budapest University of Technology and Economics
H-1521 Budapest, Hungary}
\email{molnarl@math.u-szeged.hu}
\urladdr{http://www.math.u-szeged.hu/\~{}molnarl}

\thanks{Ministry of Human Capacities, Hungary grant 20391-3/2018/FEKUSTRAT is acknowledged and the research was supported also by the National Research, Development and Innovation Office – NKFIH, Grant No. K115383.}

\begin{abstract}
In this paper we present some characterizations for quasi-arithmetic operator means (among them the arithmetic and harmonic means) on the positive definite cone of the full algebra of Hilbert space operators, and also for the Kubo-Ando geometric mean on the positive definite cone of a general non-commutative $C^*$-algebra. 
\end{abstract}

\maketitle

\section{Introduction}

Means play very important role in many parts of mathematics and also in several other branches of science. Most classically, means were considered for numbers but later for other objects, too. In this paper we are concerned with means on positive definite Hilbert space operators, especially with so-called quasi-arithmetic operator means and, even rather, with Kubo-Ando means. Both concepts cover in fact large classes of means and, in what follows, our aim is to present some characterizations for certain elements of them. Special attention is paid to the three most fundamental operator means: the arithmetic, harmonic and geometric means.

First we fix the notation and present the necessary definitions and information what we will use throughout the paper.

The theory of means for numbers is a classical and very well developed area in mathematical analyis, see Chapters II and III in the famous book \cite{HLP} by Hardy, Littlewood and P\'olya. One can say, very generally, that a mean on a subinterval $J$ of the real line is a (usually continuous) function $M:J\times J\to J$ with the simple property that 
$$\min\{t,s\}\leq M(t,s)\leq \max\{t,s\},\quad t,s\in J.$$ 
The most essential and important means are 
the arithmetic, harmonic and geometric means. They have a common generalization, the family of so-called quasi-arithmetic means. Let $J$ be a subinterval of $\R$ and $\varphi:J\to \R$ be a continuous injective function. Then the map $M_\varphi:J\times J\to J$ defined by
\begin{equation*}
M_\varphi(t,s)=\varphi^{-1}\ler{\frac{\varphi(t)+\varphi(s)}{2}},\quad t,s\in J
\end{equation*}
is called a quasi-arithmetic mean. Those means have a whole theory, their first systematic study was done in the classical work \cite{HLP}.

In what follows we will consider means on Hilbert space operators rather than on numbers. In the rest of the paper, let $H$ be a complex Hilbert space of dimension at least 2. Denote by $\Bh$ the algebra of all bounded linear operators on $H$ and by $\Bs$ the linear space of all self-adjoint elements of $\Bh$. The usual order on $\Bs$ (the one coming from the concept of positive semidefinitness) is denoted by $\leq$. The symbol $\Bp$ stands for the cone of all positive semidefinite operators in $\Bh$ and $\Bpp$ denotes the cone of all positive definite operators (operators which are positive semidefinite and invertible) in $\Bh$.

The probably most important notion and theory of means for matrices or linear operators is due to Kubo and Ando \cite{KubAnd80} and concerns positive (semidefinite or definite) Hilbert space operators. First, a binary operation $\sigma$ on $\Bp$ is said to be a connection if  the following requirements are fulfilled (from (a) to (c), all operators are supposed to belong to $\Bp$):

\begin{itemize}
\item[(a)] if $A\leq C$ and $B\leq D$, then $A\sigma B\leq C\sigma D$;
\item[(b)] $C(A\sigma B)C\leq (CAC)\sigma (CBC)$;
\item[(c)] if $A_n \downarrow A$ and $B_n \downarrow B$ strongly,
then $A_n\sigma B_n \downarrow A\sigma B$ strongly (the sign $\downarrow$ refers to monotone decreasing convergence meant in the usual order among self-adjoint operators).
\end{itemize}

If $I\sigma I=I$ holds too, then the connection $\sigma$ is called a Kubo-Ando mean. Operations as convex combination and order among connections are defined in a straightforward natural way.
By the celebrated result Theorem 3.2 in \cite{KubAnd80}, for infinite dimensional $H$, there is an affine order isomorphism from the class of all connections $\sigma$ on $\Bp$ onto the class of all operator monotone functions $f:]0,\infty [\to [0,\infty[$ given by the formula $f(t)I=I\sigma tI$, $t>0$. For invertible $A,B\in \Bp$, we have
\begin{equation}\label{E:funct}
A\sigma B=A^{\fel} f(A^\mfel BA^\mfel) A^{\fel}.
\end{equation}
Observe that these imply that Kubo-Ando means do not depend on the underlying spaces, they only depend on certain real functions.
By property (c) we obtain that the formula \eqref{E:funct} extends to any invertible $A\in \Bp$ and arbitrary $B\in \Bp$.
The most important Kubo-Ando means are the arithmetic mean with representing function $t\mapsto (1+t)/2$, the harmonic mean with representing function $t\mapsto (2t)/(1+t)$ and the geometric mean with representing function $t\mapsto \sqrt t$, $t>0$. For invertible $A,B\in \Bp$, those means of $A,B$ are in turn the following operators
\[
\frac{A+B}{2}, \quad 2(A^{-1}+B^{-1})^{-1},  \quad A^{\fel} (A^\mfel BA^\mfel)^{\fel} A^{\fel}.
\]

Below, whenever we write $\sigma, f$, we always mean that $\sigma$ is a Kubo-Ando mean and $f$ is its representing operator monotone function (which necessarily satisfies $f(1)=1$). We know from the theory of operator monotone functions that each such $f$ has a holomorphic extensions to the complex upper half plane. (In particular, it follows that $f$ is strictly positive and injective if not constant.) Moreover, $f$ is necessarily infinitely many times differentiable and $f$ is operator concave (in fact, this latter property is equivalent to the operator monotonicity). The transpose $\sigma'$ of the Kubo-Ando mean $\sigma$ is the mean
satisfying $A\sigma' B=B\sigma A$, $A,B\in \Bp$. Its representing function is $t\mapsto tf(1/t)$, $t>0$. The Kubo-Ando mean $\sigma$ is called symmetric if $\sigma'=\sigma$.
The adjoint $\sigma^*$ of $\sigma$ is the Kubo-Ando mean satisfying $A\sigma^* B=(A^{-1}\sigma B^{-1})^{-1}$ for all $A,B\in \Bpp$. Its representing function is $t\mapsto 1/f(1/t)$, $t>0$.

Beside the class of Kubo-Ando means, there are other types of operator means which are also important in particular because of their applications. One of them is
the Log-Euclidean mean which is defined in the following way
\begin{equation*}
\Mle(A,B)=\exp{(\log A+\log B)}, \quad A,B\in \Bpp.
\end{equation*}
This mean has very practical applications, e.g., in the field of
diffusion tensor magnetic resonance imaging (DT - MRI), see \cite{LEuclid}. Obviously, it is just a particular element of a general class of means what we may call quasi-arithmetic operator means. Namely, for a strictly monotone continuous function $\varphi:]0,\infty[\to \R$, define
\begin{equation}\label{E:QAM}
M_\varphi (A,B)=\varphi^{-1} \ler{\frac{\varphi(A)+\varphi(B)}{2}}, \quad A,B\in \Bpp.
\end{equation}
We note that here, just as above in relation with Kubo-Ando means, the image of an operator under a real function is meant in the sense of the continuous function calculus developed for normal elements in general $C^*$-algebras.

\section{Results}

We now turn to our results and first recall that our aim in this paper is to present some characterizations concerning Kubo-Ando means and quasi-arithmetic means on positive definite operators. We refer to \cite{Fujii} where an interesting characterization for the arithmetic and harmonic means can be found.
Our first statement concerns general quasi-arithmetic means. 
Assume that the function $\varphi$ in \eqref{E:QAM} maps onto $\R$. It is apparent that $M_\varphi$ has the following properties:
\begin{itemize}
\item{} $M_\varphi(A,A)=A$ is valid for all $A\in \Bpp$;
\item{} $M_\varphi (A,B)=M_\varphi (B,A)$ holds for any $A,B\in \Bpp$;
\item{} for an arbitrary given $B\in \Bpp$, the map $A\mapsto M_\varphi (A,B)$ is a bijection from $\Bpp$ onto itself and it preserves the order in both directions, i.e., 
\begin{equation*}
\varphi(A)\leq \varphi(A') \Longleftrightarrow 
\varphi(M_\varphi (A,B))\leq \varphi(M_\varphi (A',B))
\end{equation*}
holds for any $A,A'\in \Bpp$.
\end{itemize}
Now, the probably somewhat surprising fact is that the above properties completely characterize $M_\varphi$.  Observe that such a characterization would totally fail in the case of numbers (i.e., where the underlying Hilbert space is 1-dimensional) due to the fact that, in the case of increasing $\varphi$, the inequality $\varphi(t)\leq \varphi(s)$ does not tell much, only that $t\leq s$ holds (and this is independent from $\varphi$).

The precise statement of our first result reads as follows.

\begin{theorem}\label{T:QA}
Let $\varphi:]0,\infty [\to \R$ be a surjective continuous strictly monotone function. Assume that $M:\Bpp \times \Bpp\to \Bpp$ is a map with the following properties:
\begin{itemize}
\item[(i)] $M(A,A)=A$ is valid for all $A\in \Bpp$;
\item[(ii)] $M(A,B)=M(B,A)$ holds for any $A,B\in \Bpp$;
\item[(iii)] for any given $B\in \Bpp$, the map $A\mapsto M(A,B)$ is a bijection from $\Bpp$ onto itself with the property that
\begin{equation*}
\varphi(A)\leq \varphi(A') \Longleftrightarrow 
\varphi(M(A,B))\leq \varphi(M(A',B))
\end{equation*}
holds for any $A,A'\in \Bpp$.
\end{itemize}
Then 
we have
\begin{equation*}
M(A,B)=M_\varphi (A,B)=\varphi^{-1}\ler{\frac{\varphi(A)+\varphi(B)}{2}}, \quad A,B\in \Bpp.
\end{equation*}
\end{theorem}

\begin{proof}
In the proof we will use ideas from the proof of Theorem 8 in \cite{ML11a} what we develop further here.
Assume that $\varphi:]0,\infty [\to \R$ is a surjective  continuous strictly monotone function, and $M:\Bpp \times \Bpp\to \Bpp$ satisfies the properties (i)-(iii) above. For an arbitrary fixed $B\in \Bpp$, set
\[
\psi_B(A)=M(A,B), \quad A\in \Bpp.
\]
We know that $\psi_B:\Bpp \to \Bpp$ is a bijection and 
\begin{equation*}
\varphi(A)\leq \varphi(A') \Longleftrightarrow \varphi(\psi_B(A))\leq \varphi(\psi_B(A'))
\end{equation*}
holds for any $A,A'\in \Bpp$. It then follows that the map $\varphi\circ \psi_B\circ \varphi^{-1}$ is an order automorphism of $\Bs$, i.e., it is a bijection which preserves the order in both directions. The structure of those maps is known and was given in Theorem 2 in \cite{ML01f}. By that result, we have an invertible bounded linear or conjugate-linear operator $X_B$ on $H$ and an element $Y_B\in \Bs$ such that
\begin{equation*}
\varphi(\psi_B(\varphi^{-1}(T)))=X_B T X_B^*+Y_B, \quad T\in \Bs,
\end{equation*}
or, equivalently, 
\begin{equation}\label{E:11}
\varphi(M(A,B))=\varphi(\psi_B(A))=X_B \varphi(A)X_B^*+Y_B, \quad A\in \Bpp.
\end{equation}
Since $\psi_B(B)=M(B,B)=B$, we get that
\begin{equation*}
\varphi(B)=X_B \varphi(B)X_B^*+Y_B
\end{equation*}
and this implies that
\begin{equation*}
Y_B=\varphi(B)-X_B \varphi(B)X_B^*.
\end{equation*}
Hence by \eqref{E:11} we have
\begin{equation}\label{E:13}
\varphi(M(A,B))=\varphi(\psi_B(A))=X_B (\varphi(A)-\varphi(B))X_B^*+\varphi(B), \quad A\in \Bpp.
\end{equation}
Selecting arbitrary $R,S\in \Bs$ and using \eqref{E:13}, we deduce
\begin{equation*}\label{E:12}
\varphi(M(\varphi^{-1}(S),\varphi^{-1}(R)))=X_{\varphi^{-1}(R)} (S-R)X_{\varphi^{-1}(R)}^*+R.
\end{equation*}
Applying this, for any $R,S,T\in \Bs$ we compute
\begin{equation}\label{E:14}
\begin{gathered}
\varphi\ler{M\ler{\varphi^{-1}\ler{\frac{S+T}{2}},\varphi^{-1}(R)}}
=X_{\varphi^{-1}(R)} \ler{ \frac{S+T}{2}-R} X_{\varphi^{-1}(R)}^*+R
\\=
\frac{1}{2} \ler{( X_{\varphi^{-1}(R)} \bigl(S -R\bigr)X_{\varphi^{-1}(R)}^*+R)+( X_{\varphi^{-1}(R)} \bigl(T -R\bigr)X_{\varphi^{-1}(R)}^*+R)}\\
=\frac{1}{2}(\varphi(M(\varphi^{-1}(S),\varphi^{-1}(R))+\varphi(M(\varphi^{-1}(T),\varphi^{-1}(R))).
\end{gathered}
\end{equation}
By the symmetry property (ii) of $M$, we can continue this as
\begin{equation}\label{E:15}
\begin{gathered}
=
\frac{1}{2}(\varphi(M(\varphi^{-1}(R),\varphi^{-1}(S))+\varphi(M(\varphi^{-1}(R),\varphi^{-1}(T)))\\=
\frac{1}{2} \ler{( X_{\varphi^{-1}(S)} \bigl(R -S\bigr)X_{\varphi^{-1}(S)}^*+S)+( X_{\varphi^{-1}(T)} \bigl(R -T\bigr)X_{\varphi^{-1}(T)}^*+T)}.
\end{gathered}
\end{equation}
Similarly, we have
\begin{equation}\label{E:16}
\begin{gathered}
\varphi\ler{M\ler{\varphi^{-1}\ler{\frac{S+T}{2}},\varphi^{-1}(R)}}=
\varphi\ler{M\ler{\varphi^{-1}(R),\varphi^{-1}\ler{\frac{S+T}{2}}}}\\=
X_{\varphi^{-1}\ler{\frac{S+T}{2}}} \ler{R-\frac{S+T}{2}}X_{\varphi^{-1}\ler{\frac{S+T}{2}}}^*+\frac{S+T}{2}.
\end{gathered}
\end{equation}
Comparing \eqref{E:14}, \eqref{E:15} and \eqref{E:16}, we obtain that
\begin{equation*}
\begin{gathered}
\frac{1}{2} \ler{(X_{\varphi^{-1}(S)} \bigl(R -S\bigr)X_{\varphi^{-1}(S)}^*)+( X_{\varphi^{-1}(T)} \bigl(R -T\bigr)X_{\varphi^{-1}(T)}^*)}\\=
X_{\varphi^{-1}\ler{\frac{S+T}{2}}} \ler{R-\frac{S+T}{2}}X_{\varphi^{-1}\ler{\frac{S+T}{2}}}^*.
\end{gathered}
\end{equation*}
Fixing $S,T\in \Bs$ and writing $R+S$ in the place of $R$, we have
\begin{equation}\label{E:17}
 X_{\varphi^{-1}(S)} R X_{\varphi^{-1}(S)}^*+ X_{\varphi^{-1}(T)} (R +(S-T))X_{\varphi^{-1}(T)}^*=
X_{\varphi^{-1}\ler{\frac{S+T}{2}}} (2R+ (S-T))X_{\varphi^{-1}\ler{\frac{S+T}{2}}}^*.
\end{equation}
Choosing $R=0$, we see
\begin{equation*}
X_{\varphi^{-1}(T)} (S-T)X_{\varphi^{-1}(T)}^*=
X_{\varphi^{-1}\ler{\frac{S+T}{2}}} (S-T)X_{\varphi^{-1}\ler{\frac{S+T}{2}}}^*
\end{equation*}
and then we obtain from \eqref{E:17} that
\begin{equation*}
X_{\varphi^{-1}(S)} R X_{\varphi^{-1}(S)}^*+ X_{\varphi^{-1}(T)} R X_{\varphi^{-1}(T)}^*=
X_{\varphi^{-1}\ler{\frac{S+T}{2}}} (2R) X_{\varphi^{-1}\ler{\frac{S+T}{2}}}^*
\end{equation*}
holds for all $R\in \Bs$. This identity implies a strong connection between $X_{\varphi^{-1}(S)},  X_{\varphi^{-1}(T)}$.
Indeed, by Lemma 7 in \cite{ML11a}, it follows that 
$X_{\varphi^{-1}(S)},  X_{\varphi^{-1}(T)}$ are necessarily linearly dependent for any $S,T\in \Bs$.
This gives us that there is an invertible bounded linear or conjugate-linear operator $X$ on $H$ and a scalar valued function $t:\Bpp \to \mathbb C$ such that
\begin{equation*}
X_A=t(A)X, \quad A\in \Bpp.
\end{equation*}
Therefore, with $\tau(B)=|t(B)|^2>0$, $B\in \Bpp$ and applying \eqref{E:13}, we obtain that
\begin{equation}\label{E:21}
\varphi(M(A,B))=\tau(B)X(\varphi(A)-\varphi(B))X^*+\varphi(B), \quad A,B\in \Bpp.
\end{equation}
Using the symmetry property (ii) of $M$, we have
\begin{equation*}
\tau(B)X(\varphi(A)-\varphi(B))X^*+\varphi(B)=
\tau(A)X(\varphi(B)-\varphi(A))X^*+\varphi(A), \quad A,B\in \Bpp.
\end{equation*}
This yields
\begin{equation}\label{E:20}
(\tau(A)+\tau(B))X(\varphi(A)-\varphi(B))X^*=
\varphi(A)-\varphi(B), \quad A,B\in \Bpp.
\end{equation}
The operator $\varphi(A)-\varphi(B)$ can be equal to any rank-one element of $\Bs$. It then follows easily that for every vector $h\in H$, $Xh$ is a scalar multiple of $h$. This readily implies that $X$ is necessarily a constant multiple of the identity, i.e., $X=cI$ holds for some non-zero $c\in \mathbb C$.
We then deduce from \eqref{E:20} that $(\tau(A)+\tau(B))|c|^2=1$ is valid for any pair of different elements $A,B\in \Bpp$. 
Moreover, we learn from \eqref{E:11} that 
\begin{equation*}
M(A,B)=\varphi^{-1}(X_B \varphi(A)X_B^*+Y_B), \quad A\in \Bpp
\end{equation*}
which shows that $M$ is continuous in its first variable implying, by symmetry, that it is continuous also in its second variable. It follows from \eqref{E:21} that 
\begin{equation}\label{E:21a}
\varphi(M(A,B))=\tau(B)|c|^2(\varphi(A)-\varphi(B))+\varphi(B), \quad A,B\in \Bpp
\end{equation}
and then we can easily deduce that
$\tau$ is continuous (in the operator norm topology). Consequently,
$(\tau(A)+\tau(B))|c|^2=1$ holds for all $A,B\in \Bpp$ (and not only for different $A,B$ what we already know). It follows that $\tau(B)|c|^2=1/2$, $B\in \Bpp$, and then we obtain from  \eqref{E:21a} that 
\begin{equation*}
\varphi(M(A,B))=\frac{1}{2}\ler{\varphi(A)-\varphi(B)} +\varphi(B),
\end{equation*}
or, equivalently, that
\begin{equation*}
M(A,B)=\varphi^{-1}\ler{\frac{\varphi(A)+\varphi(B)}{2}}, \quad A,B\in \Bpp.
\end{equation*}
This completes the proof of the theorem. 
\end{proof}

As a trivial corollary, we have the following characterization of the Log-Euclidean mean. We recall that the chaotic order $\ll$ on $\Bpp$ is defined as follows: for any $A,B\in \Bpp$, we have $A\ll B$ if and only if $\log A\leq \log B$.

\begin{corollary}\label{C:QA}
Assume that $M:\Bpp \times \Bpp\to \Bpp$ is a map with the following properties:
\begin{itemize}
\item[(i)] $M(A,A)=A$ for any $A\in \Bpp$;
\item[(ii)] $M(A,B)=M(B,A)$ holds for all $A,B\in \Bpp$;
\item[(iii)] for any given $B\in \Bpp$, the transformation $A\mapsto M(A,B)$ is a bijective map from $\Bpp$ onto itself and it is a chaotic order automorphism meaning that
\begin{equation*}
A\ll A' \Longleftrightarrow 
M(A,B)\ll M(A',B)
\end{equation*}
holds for all $A,A'\in \Bpp$.
\end{itemize}
Then we have
\begin{equation*}
M(A,B)=\exp\ler{\frac{\log A+\log B}{2}}, \quad A,B\in \Bpp.
\end{equation*}
\end{corollary}

For the sake of curiosity, we present the following negative result. We once again recall that in this paper $H$ is a Hilbert space of dimension at least 2.

\begin{proposition}\label{P:nonex}
There is no such map $M:\Bpp \times \Bpp\to \Bpp$ which has the following properties:
\begin{itemize}
\item[(i)] $M(A,A)=A$ for any $A\in \Bpp$;
\item[(ii)] $M(A,B)=M(B,A)$ holds for all $A,B\in \Bpp$;
\item[(iii)] for any given $B\in \Bpp$, the transformation $A\mapsto M(A,B)$ is a bijective map from $\Bpp$ onto itself which is an order automorphism, i.e.,
\begin{equation*}
 A\leq  A' \Longleftrightarrow 
 M(A,B)\leq M(A',B)
\end{equation*}
holds for any $A,A'\in \Bpp$.
\end{itemize}
\end{proposition}

This result may look a bit surprising in the view of the fact that in the one-dimensional case (where $\Bpp$ is identified with the positive half-line) maps which satisfy the above requirements trivially exist, consider, e.g., $M(t,s)=\sqrt{ts}$, $t,s>0$.

\begin{proof}[Proof of Proposition \ref{P:nonex}]
In the light of the proof of Theorem \ref{T:QA}, our argument  is easy. Assume on the contrary that $M$ with the prescribed properties in the proposition does exist. Then, just as in the proof of Theorem \ref{T:QA}, for any fixed  $B\in \Bpp$, we set
\[
\psi_B(A)=M(A,B), \quad A\in \Bpp.
\]
It follows that $\psi_B$ is an order automorphism of $\Bpp$ with respect to the usual order. The structure of such maps was given in \cite{ML11a}. Theorem 1 there tells us that we have an invertible bounded linear or conjugate-linear operator $X_B$ on $H$ such that
\begin{equation*}
M(A,B)=\psi_B(A)=X_B A X_B^*, \quad A\in \Bpp.
\end{equation*}
This means that $M$ is additive in its first (and, by the symmetry, also in its second) variable. Now, computing
\[
M(C,A+B)=M(A+B,C)=M(A,C)+M(B,C)=M(C,A)+M(C,B), \quad A,B,C\in \Bpp,
\]
we obtain that for any given $A,B\in \Bpp$, the equality
\[
X_{A+B}CX_{A+B}=X_A C {X_A}^*+X_B C {X_B}^*
\]
holds for all $C\in \Bpp$ and then it holds also for all elements of $\Bs$ ($\Bpp$ linearly generates $\Bs$). As in the proof of Theorem \ref{T:QA}, we obtain that $X_A$, $X_B$ are linearly dependent for all $A,B\in \Bpp$. We then have an invertible bounded linear or conjugate-linear operator $X$ on $H$ and a scalar function $t$ on $\Bpp$ with positive values such that
\[
M(A,B)=t(B)XAX^*, \quad A,B\in \Bpp.
\]
By the symmetry of $M$, we have $t(B)A=t(A)B$, i.e., $A,B$ are linear dependent for all $A,B\in \Bpp$ which is an obvious  contradiction. This completes the proof of the statement.
\end{proof}

Observe that the key tools in the proofs of our statements Theorem \ref{T:QA} and Proposition \ref{P:nonex} were the  structural results on the (non-linear) order automorphisms of $\Bs$ and $\Bpp$, respectively.

In the rest of the paper we deal with Kubo-Ando means. At this point, it is very natural to ask how much do those means  differ from the previously considered quasi-arithmetic operator means. The elements of the intersection of those two families can be easily identified. Namely,
the only Kubo-Ando means which are also quasi-arithmetic operator means are the arithmetic mean and the harmonic mean. In fact, this follows easily from the much stronger result Corollary \ref{C:10} below.

One of the ideas behind quasi-arithmetic means is that in that way we can create new and useful means from the so fundamental and simple arithmetic mean. How does this idea would work for other Kubo-Ando means in the place of the arithmetic mean? Could it happen that, in the case of a given mean, we cannot produce any new Kubo-Ando mean in that way? The following result and the proceeding remark show that these questions can lead to a characterization of the geometric mean as well as to a characterization of commutative $C^*$-algebras. As for the possibly  unexpected appearance of $C^*$-algebras here, we recall that any $C^*$-algebra is isometrically and isomorphically embedded into some Hilbert space operator algebra $B(H)$ as a normed closed *-subalgebra containing the identity (in this paper any $C^*$-algebra is assumed to be unital). Therefore, any $C^*$-algebra can be viewed as a subalgebra of some $\Bh$ and hence it has sense to speak about Kubo-Ando means on the positive definite cone of an arbitrary $C^*$-algebra. (We also recall that in the introduction we already mentioned that Kubo-Ando means do not depend on underlying spaces, they depend only on operator monotone real functions.) If $\A$ is a $C^*$-algebra, we denote by $\App$ its positive definite cone (set of all positive invertible elements in $\A$). Our result reads as follows.

\begin{theorem}\label{T:Ca}
Let $\sigma$ be a symmetric Kubo-Ando mean. For a non-commutative $C^*$-algebra $\A$, the following conditions are equivalent:
\begin{itemize}
\item[(i)] $\sigma$ is the geometric mean;
\item[(ii)] there is no Kubo-Ando mean $\tau$ different from $\sigma$ of the form
\begin{equation}\label{E:veg}
A\tau B=\varphi^{-1}(\varphi(A)\sigma \varphi(B)),\quad A,B\in \App 
\end{equation}
with some continuous strictly monotone function $\varphi:]0,\infty[\to ]0,\infty[$.
\end{itemize} 
\end{theorem}

In the proof of the theorem we will use the following characterization of the usual order which was given in \cite{CMM}. In what follows $\sharp$ stands for the Kubo-Ando geometric mean.

\begin{lemma}\label{L:3}
Let $\A$ be a $C^*$-algebra and $A,B\in \App$. We have $A\leq B$ if and only if $\| A\sharp X\|\leq \|B\sharp X\|$ holds for all $X\in \App$.
\end{lemma}

\begin{proof}[Proof of Theorem \ref{T:Ca}]
Assume that (ii) holds for the symmetric Kubo-Ando mean $\sigma$ on $\App$. Then choosing the function $\varphi(t)=1/t$, $t>0$ which produces the adjoint Kubo-Ando mean, we have that
\begin{equation*}
(A^{-1}\sigma B^{-1})^{-1}=A\sigma B, \quad  A,B\in \App. 
\end{equation*}
This implies that for the representing operator monotone function $f:]0,\infty [ \to ]0,\infty [$ of $\sigma$, we have 
\begin{equation*}
\frac{1}{f(\frac{1}{t})}=f(t), \quad t>0.
\end{equation*}
By the symmetry of $\sigma$, we also have $tf(1/t)=f(t)$, $t>0$. From these we deduce that
\begin{equation*}
f(t)^2=t, \quad t>0,
\end{equation*}
which implies that $f(t)=\sqrt{t}$, $t>0$, i.e., $\sigma$ is the geometric mean.

Now assume that there is a continuous strictly monotone function $\varphi:]0,\infty [ \to ]0,\infty[$ such that 
\begin{equation}\label{E:63}
A\tau B=\varphi^{-1}(\varphi(A)\sharp \varphi(B)),\quad A,B\in \App 
\end{equation}
is a Kubo-Ando mean. Let $f$ be the operator monotone function representing $\tau$. Plugging $A=tI, B=sI$, $t,s>0$ into \eqref{E:63} and applying the formula \eqref{E:funct}, we have
\begin{equation*}
\varphi^{-1}(\sqrt{\varphi(t)\varphi(s)})=tf(s/t), \quad t,s>0.
\end{equation*}
Defining $\psi=\log \circ \varphi$, we obtain 
\begin{equation*}
\psi^{-1}\ler{\frac{\psi(t)+\psi(s)}{2}}=tf(s/t), \quad t,s>0.
\end{equation*}
It follows that the numerical quasi-arithmetic mean on the left hand side is homogeneous. It is a famous fact, see \textbf{84.} on page 68 in \cite{HLP}
or Theorem 2 on page 153 in \cite{Acz66}, that this implies that $\psi$ is either of the form 
\begin{equation}\label{E:62}
\psi(t)=a t^p +b, \quad t>0,
\end{equation} 
with some non-zero exponent $p$ and constants $0\neq a\in \R$, $b\in \R$, or it is of the form 
\begin{equation*}
\psi(t)=a \log t +b, , \quad t>0,
\end{equation*}
where $a,b$ are constants with the same properties as above.
Simple calculation shows that, concerning $f$, in the first case we have  
\begin{equation}\label{E:61}
f(t)=\ler{\frac{1+t^p}{2}}^{1/p}, \quad t>0,
\end{equation}
while in the latter case we have 
\begin{equation*}
f(t)=\sqrt{t} , \quad t>0.
\end{equation*}
In addition, we know that the function $f$ is operator monotone. For the square-root function this is true. But as for the former parametric family of functions in \eqref{E:61}, 
by Theorem 4 in \cite{BesPet}, we have operator monotonicity exactly when the non-zero exponent $p$ satisfies $-1\leq p\leq 1$.
Assume we have \eqref{E:61} and \eqref{E:62}. Then we obtain
\begin{equation*}
\varphi(t)=\exp(a t^p +b), \quad t>0
\end{equation*}
and hence, using 
\begin{equation*}
\varphi(A\tau B)=\varphi(A)\sharp \varphi(B), \quad A,B\in \App
\end{equation*}
(this follows from \eqref{E:63}), we infer 
\begin{equation}\label{E:81}
\begin{gathered}
\exp\ler{ a \ler{{A^{\fel}}\ler{\frac{I+({A}^{-\fel}B{A}^{-\fel})^p}{2}}^{1/p}{A^{\fel}}}^p +bI}\\
=\exp\ler{a A^p+bI}\sharp \exp\ler{a B^p +bI},  \quad A,B\in \App.
\end{gathered}
\end{equation} 
Multiplying both sides of this equality by $\exp(-bI)$, we deduce easily that  \eqref{E:81} holds with $b=0$. Next, if necessary, taking inverses on both sides of  \eqref{E:81}, we can assume that $a>0$. Finally, replacing $A,B$ by $a^{-1/p} A, a^{-1/p} B$, respectively, we obtain that
\begin{equation}\label{E:82}
\begin{gathered}
\exp\ler{  \ler{{A}^{\fel}\ler{\frac{I+({A}^{-\fel}B{A}^{-\fel})^p}{2}}^{1/p}{A}^{\fel}}^p}\\
=\exp(A^p)\# \exp(B^p),  \quad A,B\in \App.
\end{gathered}
\end{equation} 
Select arbitrary $A,A'\in \App$ such that $A\leq A'$. Since the log of the left hand side of the equality \eqref{E:82} (as a function of the variables $A,B$) is a Kubo-Ando mean, we infer that
\begin{equation}\label{E:83}
\log \ler{\exp (A^p)\sharp \exp(X)} \leq \log \ler{ \exp (A'^p)\sharp \exp(X)}, \quad X\in \App.
\end{equation} 
The operators $\exp (A^p)\sharp \exp(X), \exp (B^p)\sharp \exp(X)$ are greater than or equal to $I$, hence from \eqref{E:83} we can deduce that
\begin{equation*}
\begin{gathered}
\log \| \exp (A^p)\sharp \exp(X)\|=\| \log(\exp (A^p)\sharp \exp(X))\|   \\ \leq \| \log(\exp (A'^p)\sharp \exp(X))\|= \log \| \exp (A'^p)\sharp \exp(X)\|, \quad X\in \App,
\end{gathered}
\end{equation*}
which gives
\begin{equation}\label{E:ve}
\|\exp (A^p)\sharp \exp(X)\|  \leq \| \exp (A'^p)\sharp \exp(X)\|, \quad X\in \App.
\end{equation}
It is easy to see that any element of $\App$ is a positive scalar multiple of an element of the form $\exp(X)$ with some $X\in \App$. It then follows from \eqref{E:ve} that we have
\begin{equation*}
\|\exp (A^p)\sharp Y\|  \leq \| \exp (A'^p)\sharp Y\|, \quad Y\in \App.
\end{equation*}
Applying Lemma \ref{L:3}, we conclude that $\exp (A^p)\leq \exp (A'^p)$. Therefore, we have proved that the function $t\mapsto \exp(t^p)$ is operator monotone on $\App$. Since, by Theorem 2 in \cite{NUW}, the existence of a non-concave operator monotone function on $\App$ implies the commutativity of $\A$, we arrive at a contradiction. Therefore, only the possibility that $f(t)=\sqrt{t}$, $t>0$ remains implying that $\tau$ is the geometric mean. This completes the proof of the theorem.
\end{proof}

Observe that
for a commutative $C^*$-algebra $\A$, the conditions (i) and (ii) in the above theorem are not equivalent. Indeed, $\A$ then can be assumed to be the algebra $C(K)$ of all continuous complex valued continuous functions on a compact Hausdorff space $K$. The continuous function calculus of a given element $f\in C(K)$ is the composition of continuous complex valued functions defined on the range of $f$ with $f$. Define $\varphi(t)=\exp(t)$, $t>0$. Then $\varphi$ transforms the geometric mean to the arithmetic mean, i.e., 
\[
\varphi^{-1}(\sqrt{\varphi(g)\phi(h)})=\frac{g+h}{2}, \quad g,h\in  C(K)^{++}.
\]
Therefore, Theorem \ref{T:Ca} can be used to characterize commutativity of $C^*$-algebras: On the positive definite cone of the $C^*$-algebra $\A$, the geometric mean can be transformed by the formula \eqref{E:veg} to a different Kubo-Ando mean if and only if $\A$ is commutative.

In the remaining part of the paper we present kind of algebraic characterizations of the arithmetic and harmonic means as Kubo-Ando means having certain operational properties. First observe that each of those means has the property that plugging it into some particular scalar function, we arrive at an associative operation (in the former case this particular function is $t\mapsto 2t$ while in the latter case it is $t\mapsto (1/2)t$, $t>0$). More generally, for any quasi-arithmetic operator mean 
\begin{equation*}
M_\varphi (A,B)=\varphi^{-1} \ler{\frac{\varphi(A)+\varphi(B)}{2}}, \quad A,B\in \Bpp
\end{equation*}
with continuous strictly monotone function $\varphi:]0,\infty[\to ]0,\infty[$, the assignment $(A,B)\mapsto 2\varphi(M_\varphi (A,B))$ defines an associative operation on $\Bpp$. Below we prove that for Kubo-Ando means, the property that a certain continuous strictly increasing function of the mean results in an associative operation on $\Bpp$ characterizes the arithmetic and the harmonic means. We already refered to that the double of the arithmetic mean and the half of the harmonic mean are associative operations on $\Bpp$. (We mention that in the paper \cite{Ando76}, Nishio and Ando used the associativity as a condition in a characterization of the parallel sum (which is the half of the harmonic mean when defined on the whole positive semidefinite cone $\Bp$)).
One may ask how this can happen, why the geometric mean does not show up. Indeed, for numbers, the square of the geometric mean is an associative operation. However, this is no longer true in non-commutative $C^*$-algebras as demonstrated in Proposition 6 in \cite{ML17e}.  

In the next statement we consider even a weaker form of associativity, where we do not have three independent variables, only two (the third one is fixed to the identity). One of the reasons for studying this more general setting is that it leads to a probably interesting problem. 

\begin{theorem}\label{T:ASS}
Let $\sigma$ be a symmetric Kubo-Ando mean with representing operator monotone function $f$.
Assume that there exists a continuous strictly increasing  and surjective function $g:]0,\infty[\to ]0,\infty[$ such that the operation $\diamond: (A,B)\mapsto g(A\sigma B)$, $A,B\in \Bpp$ satisfies 
\begin{equation}\label{E:65}
(A\diamond I) \diamond B= A\diamond (I \diamond B), \quad A,B\in \Bpp.
\end{equation}
Then either we have $g(f(t))=t$, $t>0$ meaning that $A\diamond I=I\diamond A=A$, $A\in \Bpp$ and hence \eqref{E:65} becomes the triviality $A\diamond B=A\diamond B$, $A,B\in \Bpp$, or we have one of the following three possibilities:
\begin{itemize}
\item[(a)] there is a positive scalar $c\neq 1$ such that
$f(c^2 t)=cf(t)$, $t>0$;
\item[(b)] $\sigma$ is the arithmetic mean;
\item[(c)] $\sigma$ is the harmonic mean.
\end{itemize}
\end{theorem}

The proof of this result rests on a few auxiliary statements. The first of them is the following.

\begin{lemma}\label{L:1}
Let $\varphi:]0,\infty[\to ]0,\infty[$ be a continuous strictly monotone increasing and surjective function. Assume that for any $A,A'\in \Bpp$ we have $A\leq A'$ if and only if $\varphi(A)\leq \varphi(A')$. Then $\varphi(t)=ct$, $t>0$ holds for some positive constant $c$. 
\end{lemma}

\begin{proof}
We know that $A\mapsto \varphi(A)$ is an order automorphism of $\Bpp$. We recall that it follows from Theorem 1 in \cite{ML11a} that there is a bounded invertible either linear or conjugate-linear operator $X$ on $H$ such that
\begin{equation}\label{E:35}
\varphi(A)=XAX^*, \quad A\in \Bpp.
\end{equation}
Let us extend $\varphi$ to $[0,\infty[$ with $\varphi(0)=0$. Clearly, the extension, what we denote by the same symbol $\varphi$, is continuous on $[0,\infty[$.
By taking limit in \eqref{E:35}, we have that 
\[
\varphi(1)P=\varphi(P)=XPX^*
\]
holds for every rank-one propejction $P$ on $H$. This implies that for any unit vector $h\in H$, the vector $Xh$ is a scalar multiple of $h$. It gives us that $X$ is a scalar multiple of the identity which completes the proof.
\end{proof}

Just a quick remark. One may think that it follows from the conditions of the previous lemma that $\varphi$ is operator monotone and hence one can use the theory of those functions to immediately get the desired conclusion. However, that argument is not correct since $H$ is a given Hilbert space which is not necessarily infinite dimensional.

Beside the above auxiliary result, we will also need the concept  of strength functions and some of their properties.
Strength functions are certain functions associated to positive semidefinite Hilbert space operators. The concept was introduced by Busch and Gudder in the paper \cite{BusGud99}.  Denote by $\PI$ the set of all rank-one projections on $H$. For any $A\in \Bp$ define
\begin{equation*}
\l(A,P)=\sup \{t\geq 0\, :\, tP\leq A\}, \quad P\in \PI.
\end{equation*}
The function $\l(A,\cdot )$ is called the strength function of $A\in \Bp$. 
By Theorem 1 in \cite{BusGud99}, the assignment
$A\longmapsto \l(A,\cdot )$
is a one-to-one correspondence between $\Bp$ and the collection of all strength functions on $\PI$ which preserves the order in both directions. Here, on $\Bp$ we consider the usual order (coming from positive semidefiniteness) while on the collection of strength functions we consider the usual pointwise order between functions ($\l(A,\cdot )\leq \l(B,\cdot )$ if and only if $\l(A,P)\leq \l(B,P)$ holds for all $P\in \PI$). Therefore, for any $A,B\in \Bp$ we have $A\leq B$ if and only if $\l(A,\cdot )\leq \l(B,\cdot )$.
Strength functions appear also in relation with Kubo-Ando means. Indeed, in Lemma 2.6 in \cite{ML11g} we proved that for any symmetric Kubo-Ando mean $\sigma$ with representing operator monotone function $f$ which satisfies $f(0)=0$, the equality
\begin{equation}\label{E:36}
A\sigma P=P\sigma A=f(\l (A,P)) P
\end{equation}
holds for any $A\in \Bp$ and $P\in \PI$. We will use this formula several times in what follows.

We now turn to the proof of our theorem above. 
We remark that in what follows, some monotone functions might only be defined on $]0,\infty[$, originally not defined at $0$. If, in spite of this, we still write the value of such a  function at $0$, we always mean its limit at $0$.

\begin{proof}[Proof of Theorem \ref{T:ASS}]
Assume first that $f(0)=0$.
Let $h$ be the composite function $h=g\circ f$. Clearly $g(0)=0$, so we have $h(0)=0$, and since $f,g$ are strictly increasing, it follows that $h(1)\neq 0$.
Let $A=P$ be any rank-one projection on $H$ and $B\in \Bpp$ be arbitrary. We compute the two sides of the equality \eqref{E:65}. 
As for the left hand side, using \eqref{E:36}, we have
\begin{equation*}
\begin{gathered}
(P\diamond I)\diamond B=(I\diamond P)\diamond B= h(P)\diamond B=(h(1)P)\diamond B\\=
g\ler{h(1) \ler{P\sigma \ler{\frac{1}{h(1)} B}}}=
g\ler{h(1) f\ler{\frac{1}{h(1)}\lambda(B,P)}}P.
\end{gathered}
\end{equation*}
The right hand side of \eqref{E:65} can be computed as follows
\begin{equation*}
P\diamond (I \diamond B)=P\diamond h(B)=g(f(\lambda (h(B),P)))P.
\end{equation*}
Since the two sides are equal, we obtain
\begin{equation*}\label{E:66}
g\ler{h(1) f\ler{\frac{1}{h(1)}\lambda(B,P)}}=g(f(\lambda (h(B),P)))
\end{equation*}
and hence that
\begin{equation}\label{E:66a}
h(1) f\ler{\frac{1}{h(1)}\lambda(B,P)}=f(\lambda (h(B),P)).
\end{equation}
Moreover, $f$ is strictly monotone increasing, hence
it follows from \eqref{E:66a} that for any given $B,B'\in \Bpp$ and for all rank-one projection $P\in \PI$ we have
\begin{equation*}
\lambda(B,P)\leq \lambda(B',P) \Leftrightarrow
\lambda(h(B),P)\leq \lambda(h(B'),P).
\end{equation*}
By the mentioned properties of strength functions, we conclude that for any given $B,B'\in \Bpp$, the inequality 
$B\leq B'$ is equivalent to $h(B)\leq h(B')$. 

Assume that $h$ maps $]0,\infty[$ onto itself. Then it follows from Lemma \ref{L:1} that $h(t)=ct$, $t>0$ holds with the positive constant $c=h(1)$. If $c=1$, then we have the trivial case $g(f(t))=t$, $t> 0$ described in the theorem. If $c\neq 1$, then from \eqref{E:66a} we easily obtain that
\begin{equation*}
cf(t/c)=f(ct), \quad t>0
\end{equation*}
and this leads to the possibility (a).

Assume now that $h$ maps $]0,\infty[$ onto a finite interval $]0,\beta[$, $\beta <\infty$. 
Then for the function
$k: t\mapsto \frac{1}{\beta -h(t)}-\frac{1}{\beta}$ we have that $A\mapsto k(A)$ is a bijection of $\Bpp$ which preserves the order in both directions, i.e., for any $A,A'\in \Bpp$ we have $A\leq A'$ if and only if $k(A)\leq k(A')$. Therefore, by Lemma \ref{L:1} again, $k$ is a positive scalar multiple of the identity. This means that
\begin{equation*}
\frac{1}{\beta -h(t)}-\frac{1}{\beta}=ct, \quad t>0
\end{equation*}
holds with some positive constant $c$. Simple calculation shows that, for $h$ itself, the equality above yields
\begin{equation}\label{E:52}
g(f(t))=h(t)=\frac{\beta^2ct}{\beta ct+1}, \quad t>0.
\end{equation}

Consider the equality \eqref{E:65} now for $A=tI, B=sI$, $t,s>0$. We have
\begin{equation*}
\begin{gathered}
g\ler{h(t)f\ler{\frac{s}{h(t)}}}=
g\ler{tf\ler{\frac{h(s)}{t}}}
\end{gathered}
\end{equation*}
and then, by the injectivity of $g$, we have
\begin{equation*}
\begin{gathered}
h(t)f\ler{\frac{s}{h(t)}}=
tf\ler{\frac{h(s)}{t}}
\end{gathered}
\end{equation*}
for all $t,s>0$.
By the symmetry of the Kubo-Ando mean $\sigma$ we have $tf(1/t)=f(t)$, $t>0$. Hence,
\[
h(t)f\ler{\frac{s}{h(t)}}=sf\ler{\frac{h(t)}{s}}
\]
and thus we obtain
\begin{equation*}
\begin{gathered}
sf\ler{\frac{h(t)}{s}}=
tf\ler{\frac{h(s)}{t}}, \quad t,s>0.
\end{gathered}
\end{equation*}
Define the function $F$ by $F(t)=1/(f(1/t))$, $t>0$. Simple calculation shows that
\begin{equation*}
\begin{gathered}
sF\ler{\frac{t}{h(s)}}=
tF\ler{\frac{s}{h(t)}} , \quad t,s>0.
\end{gathered}
\end{equation*}
Using \eqref{E:52}, we obtain
\begin{equation*}
\begin{gathered}
sF\ler{\frac{t(\beta cs+1)}{\beta^2 cs}}=
tF\ler{\frac{s(\beta ct+1)}{\beta^2 ct}}, \quad t,s>0.
\end{gathered}
\end{equation*}
Denoting $x=s/t$, we deduce
\begin{equation*}
\begin{gathered}
xF\ler{\frac{\beta cs+1}{\beta^2 cx}}=
F\ler{\frac{\beta cs+x}{\beta^2 c}},\quad x,s>0.
\end{gathered}
\end{equation*}
Since $f$ is operator monotone, it is infinitely many times differentiable and hence the same is true for $F$. In fact, $F$ is the representing operator monotone function of the adjoint of $\sigma$. Differentiating in the equality above with respect to the variable $s$, we obtain
\begin{equation}\label{E:67}
\begin{gathered}
F'\ler{\frac{\beta cs+1}{\beta^2 cx}}=
F'\ler{\frac{\beta cs+x}{\beta^2 c}},\quad x,s>0.
\end{gathered}
\end{equation}
Since $F$ is operator monotone, it is necessarily operator concave implying that $F'$ is monotone decreasing. We then deduce from \eqref{E:67} that $F'$ is constant on the interval bordered by the points
\begin{equation*}
\frac{\beta cs+1}{\beta^2 cx},\frac{\beta cs+x}{\beta^2 c}.
\end{equation*}
Fixing $s$ and letting $x$ tend to infinity, we see that the former end-points converge to 0 while the latter end-points converge to infinity. This means that $F'$ is constant on the whole interval $]0,\infty[$. Consequently, $F(t)=dt+e$, $t>0$ holds with some constants $d>0$, $e\geq 0$.
We then infer that $f$ is of the form 
\begin{equation*}
f(t)=\frac{t}{et+d}, \quad t>0.
\end{equation*}
By the symmetry of the Kubo-Ando mean $\sigma$, it follows that $e=d$ and then, using the fact that $f(1)=1$, we have 
\begin{equation*}
f(t)=\frac{2t}{1+t}, \quad t>0.
\end{equation*}
Therefore, $\sigma$ is the harmonic mean which is the possibility (c) in the theorem.

Assume now that $f(0)>0$. Then for the operator monotone function
\begin{equation*}
F(t)=1/(f(1/t))=t/f(t), \quad t>0
\end{equation*}
we have $F(0)=0$. As we have noted above, $F$ is the representing operator monotone function of the adjoint $\sigma^*$ of $\sigma$,
\[
A\sigma^* B=(A^{-1}\sigma B^{-1})^{-1}, \quad A,B\in \Bpp,
\]
which is a symmetric Kubo-Ando mean.
Define $G:]0,\infty [\to ]0,\infty [$ by
\begin{equation*}
G(t)=1/(g(1/t)), \quad t>0.
\end{equation*}
Since $g$ maps onto $]0,\infty[$, it follows that $G$ is a continuous strictly increasing function which also maps $]0,\infty[$ onto itself. Moreover, the operation
$\circ : (A,B)\to G(A\sigma^* B)$, $A,B\in \Bpp$ satisfies 
\eqref{E:65}. 
Indeed, it easily follows from the observations that
\[
A\circ B=G( A\sigma^* B)=(g(A^{-1}\sigma B^{-1}))^{-1}=(A^{-1}\diamond B^{-1})^{-1} , \quad A,B\in \Bpp.
\]
Therefore, by the first part of the proof, we obtain that either $G(F(t))=t$, $t>0$ which is equivalent to $g(f(t))=t$, $t>0$, or $G(F(t))=ct$, $t>0$ holds with some positive constant $c\neq 1$ which is equivalent to that $g(f(t))=(1/c)t$, $t>0$. In this latter case we once again (as in the first part of the proof) obtain the possibility (a) with $1/c$ in the place of $c$. Otherwise, by the argument above we conclude  that $\sigma^*$ is the harmonic mean implying that the original mean $\sigma$ is the arithmetic mean. This is the possibility (b) in the theorem. The proof is now complete.
\end{proof}

We remark the following.
First, the converse of the statement in Theorem \ref{T:ASS} is also true. That means that if $\sigma$ is a symmetric Kubo-Ando mean which is either the arithmetic mean, or the harmonic mean, or the representing operator monotone function $f$ of $\sigma$ has the property that there is a positive constant $c$ different from 1 such that
\begin{equation}\label{E:c}
f(c^2t)=cf(t), \quad t>0, 
\end{equation}
then there a continuous strictly monotone increasing and surjective function $g:]0,\infty \to ]0,\infty[$ such that the operation $\diamond: (A,B)\mapsto g(A\sigma B)$, $A,B\in \Bpp$ satisfies \eqref{E:65}. Indeed, the cases of the arithmetic and harmonic means are trivial. In the remaining case we can assume that $c<1$ (indeed, if $c$ satisfies \eqref{E:c}, then $1/c$ also satisfies it). We can next deduce that $f(c^{2n}t)=c^nf(t)$ is valid for any integer $n$. This implies that the limit of $f$ is 0 at 0 and $\infty$ at $\infty$. Therefore, $f$ is a bijection of $]0,\infty[$ onto itself. Defining $g(t)=cf^{-1}(t)$, $t>0$ we get the desired function $g$. Indeed, we have $g(f(t))=ct$ and $cf((1/c)t)=f(ct)$, $t>0$, and one can check that
\[
\begin{gathered}
g(A\sigma I)\sigma B=g(f(A))\sigma B=(cA)\sigma B\\=
c {A}^{\fel}f((1/c){A}^{-\fel}B{A}^{-\fel}){A}^{\fel}=
{A}^{\fel}f(c{A}^{-\fel}B{A}^{-\fel}){A}^{\fel}\\=
A\sigma(cB)=A\sigma g(f(B))=A\sigma g(I\sigma B)
\end{gathered}
\]
holds for all $A,B\in \Bpp$
which is equivalent to \eqref{E:65}.

And now the problem that we have mentioned before the formulation of Theorem \ref{T:ASS}. It concerns the possibility (a). Assume $f:]0,\infty[\to ]0,\infty[$ is an operator monotone function with $f(0)=0$ and $f(1)=1$ which is symmetric (i.e., satisfies $tf(1/t)=f(t)$, $t>0$) and has the property that $f(c^2t)=cf(t)$, $t>0$ holds with some positive real number $c\neq 1$. Does it follow that $f$ is necessarily the square root function? We recall that operator monotone functions have very strong analytic regularity properties and the condition (a) also looks very restrictive. However, we still do not know if the answer to the question is positive or negative. If it were affirmative, then we would get an  interesting common characterization of the three fundamental operator means, the arithmetic, harmonic and geometric means.

We now turn to the 'global' associativity, i.e., we do not make any restriction on the variables. We quite easily obtain the following result characterizing the arithmetic and harmonic means.

\begin{theorem}\label{T:assoc}
Let $\sigma$ be a symmetric Kubo-Ando mean. 
Assume that there exists a continuous strictly increasing  and surjective function $g:]0,\infty[\to ]0,\infty[$ such that the operation $\diamond: (A,B)\mapsto g(A\sigma B)$, $A,B\in \Bpp$ 
is associative. Then $\sigma$ is either the arithmetic mean or the harmonic mean.
\end{theorem}

\begin{proof}
In view of the proof of Theorem \ref{T:ASS} above, what we have to do is to rule out the possibility that $g(f(t))=ct$, $t>0$ holds with some positive scalar $c$, where $f$ is the operator monotone function representing $\sigma$.

Define the function $K:]0,\infty [\times ]0,\infty[\to ]0,\infty[$ by $K(t,s)I=(tI) \sigma (sI)$, $t,s>0$. Clearly, $K$ is strictly increasing in both of its variables and satisfies the so-called associativity equation
\begin{equation*}
K(K(t,s),r)=K(t,K(s,r)),  \quad t,s>0.
\end{equation*}
The solution of this equation is well-known.
Indeed, it follows from the theorem on page 256 in \cite{Acz66} (or see Result 11.2 on page 472 in \cite{Kan}) that
$K$ is necessarily of the form
\begin{equation*}
K(t,s)=\varphi^{-1}(\varphi(t)+\varphi(s)), \quad t,s>0
\end{equation*}
with some continuous strictly monotone function $\varphi :]0,\infty [\to \R$. For any $t,s>0$, we compute
\begin{equation}\label{E:01}
g\ler{tf\ler{\frac{s}{t}}}=K(t,s)=\varphi^{-1}(\varphi(t)+\varphi(s)).
\end{equation}
With $t=s$ we have
\begin{equation}\label{E:02}
g(t)=\varphi^{-1}(2\varphi(t)), \quad t>0.
\end{equation}
Hence, by \eqref{E:01}, it follows that
\begin{equation*}
\varphi^{-1}\ler{2\varphi\ler{tf\ler{\frac{s}{t}}}}=
\varphi^{-1}(\varphi(t)+\varphi(s))
\end{equation*}
and this implies 
\begin{equation*}
2\varphi\ler{tf\ler{\frac{s}{t}}}=
\varphi(t)+\varphi(s), \quad t,s>0.
\end{equation*}
It then follows that 
\begin{equation*}
tf\ler{\frac{s}{t}}=
\varphi^{-1}\ler{\frac{\varphi(t)+\varphi(s)}{2}}, \quad t,s>0.
\end{equation*}
This equality tells us that the quasi-arithmetic mean on the right hand side is homogeneous. Just as in the proof of Theorem \ref{T:Ca}, we obtain that $\varphi$ is either of the form
\begin{equation*}
\varphi(t)=a t^p +b, \quad t>0,
\end{equation*} 
with some non-zero exponent $p$ and constants $0\neq a\in \R$, $b\in \R$, or it is of the form 
\begin{equation*}
\varphi(t)=a \log t +b, \quad t>0,
\end{equation*}
where $a,b$ are constants with the same properties as above.
Next, just as in that proof, it follows that concerning $f$ we have the following two possibilities: either
\begin{equation*}
f(t)=\ler{\frac{1+t^p}{2}}^{1/p}, \quad t>0
\end{equation*}
holds with some non-zero $p$ with $-1\leq p\leq 1$, or
\begin{equation*}
f(t)=\sqrt{t} , \quad t>0.
\end{equation*}
As for $g$, using \eqref{E:02}, we can deduce that either
\begin{equation*}
g(t)=(2t^p+(b/a))^{1/p} , \quad t>0
\end{equation*}
or 
\begin{equation*}
g(t)=t^2 \exp(b/a) , \quad t>0.
\end{equation*}
By the surjectivity of $g$, in the first case we must have $b=0$.
Therefore, concerning the composite function $g\circ f$ we have either $g(f(t))=(1+t^p)^{1/p}$, $t>0$ or $g(f(t))=t\exp(b/a)$, $t>0$. Since, on the other hand, we assumed $g(f(t))=ct$, $t>0$, the first possibility is ruled out immediately. The possibility that $f(t)=\sqrt{t}$, $g(t)=t^2 \exp(b/a),$ $t>0$ is ruled out by the fact (what we have already mentioned) that the square of the geometric mean is not associative on non-commutative $C^*$-algebras and neither so is any of its scalar multiples. This completes the proof.
\end{proof}

We now turn to another algebraic characterization of the arithmetic and harmonic means.
One can easily see that those two means satisfy the following equality
\begin{equation}\label{E:1}
(A\sigma B)\sigma (C\sigma D)=
(A\sigma C)\sigma (B\sigma D), \quad A,B,C,D\in \Bpp.
\end{equation}
Indeed, \eqref{E:1} is fulfilled by any quasi-aritmetic operator mean.
In the one-dimensional case, on any subinterval $J$ of $\R$, the above equation for an operation on $J$ in the place of the mean $\sigma$ is called the mediality (or bisymmetry) condition and it is known to characterize quasi-arithmetic means to a certain extent, see Chapter 17 in \cite{AD}. Note that also the more general pair of conditions, where in \eqref{E:1} we have $A=B$, respectively, $C=D$, were considered in \cite{AD} and called the equations of self-distributivity. It was proved there that that pair of three variable conditions also characterize quasi-arithmetic means to some extent on the reals.
 
In the next result we prove that the validity of \eqref{E:1} for symmetric Kubo-Ando operator means characterizes the arithmetic and harmonic means already on an even much more restricted domain, with two independent variables only.

\begin{theorem}\label{T:H1}
The symmetric Kubo Ando mean $\sigma$ is either the arithmetic or the harmonic mean if and only if it satisfies
\begin{equation}\label{E:71}
(A\sigma A) \sigma (I\sigma B)=
(A\sigma I)\sigma (A\sigma B), \quad A,B,\in \Bpp.
\end{equation}
\end{theorem}

\begin{proof}
Only the sufficiency requires proof.
Assume that \eqref{E:71} holds for $\sigma$. Let $f$ be the operator monotone function representing $\sigma$. 
Assume first that $f(0)=0$. We compute in a way similar to the proof of Theorem \ref{T:ASS}. Let $A=P$ be a rank-one projection on $H$ and $B\in \Bpp$ be an arbitrary element. The left hand side of \eqref{E:71} is equal to
\begin{equation*}
P\sigma (I\sigma B)=P\sigma f(B)=f(\lambda (f(B),P))P.
\end{equation*}
As for the right hand side, we compute
\begin{equation*}
(P\sigma I)\sigma (P\sigma B)=
P\sigma( f(\lambda(B,P))P)=f(f(\lambda(B,P)))P.
\end{equation*}
Therefore, it follows that
\begin{equation*}
\lambda (f(B),P)=f(\lambda(B,P)).
\end{equation*}
Since this holds for every rank-one projection $P$ on $H$, as in the proof of Theorem \ref{T:ASS}, we easily conclude that for any $B,B'\in \Bpp$, we have $B\leq B'$ if and only if
$f(B)\leq f(B')$.

We distinguish two cases. First, if $f$ maps onto $]0,\infty[$, then by Lemma \ref{L:1} it follows that $f$ is a positive scalar multiple of the identity which contradicts to the symmetric property of $f$, i.e., to the identity $tf(1/t)=f(t)$, $t>0$. Therefore, $f$ must map onto some finite interval $]0,\beta[$, $\beta<\infty$. Just as in the corresponding part of the proof of Theorem \ref{T:ASS}, we then get that $f$ is necessarily of the form 
\begin{equation*}\label{E:72}
f(t)=\frac{\beta^2ct}{\beta ct+1}, \quad t>0
\end{equation*}
with some positive scalar $c$. By the symmeric property of $f$, it follows easily that $\beta c=1$ and then from $f(1)=1$ we deduce that $\beta=2$. Therefore,
\begin{equation*}\label{E:73}
f(t)=\frac{2 t}{t+1}, \quad t>0,
\end{equation*}
which means that $\sigma$ is the harmonic mean.

Assume now that $f(0)>0$. Then, as in the last part of the proof of Theorem \ref{T:ASS},
we can consider the mean corresponding to the function $t\mapsto 1/f(1/t)$, i.e., the adjoint $\sigma^*$ of $\sigma$. Clearly, it also satisfies \eqref{E:71} and we apply the first part of the proof to conclude that $\sigma^*$ is the harmonic mean. This implies that the original mean $\sigma$ is the arithmetic mean and the proof is complete.
\end{proof}

We have already asserted (see the discussion before Theorem \ref{T:Ca}) that a symmetric Kubo-Ando mean $\sigma$ on $\Bpp$ is a quasi-arithmetic operator mean only in the case where $\sigma$ is either the arithmetic or the harmonic mean. This means that for a symmetric Kubo-Ando mean $\sigma$ on $\Bpp$, there is no continuous injective scalar valued function $\varphi$ on $]0,\infty[$ for which we have
\begin{equation}\label{E:veg2}
\varphi(A\sigma B)=\frac{\varphi(A)+\varphi(B)}{2}, \quad A,B\in \Bpp.
\end{equation}
Theorem \ref{T:H1} has a much stronger consequence stated in the next result. Namely, it says that the only Kubo-Ando means which can be transformed to the arithmetic mean by an injective transformation from $\Bpp$ into an arbitrary linear space (or, more generally, into any uniquely 2-divisible commutative semigroup) are the arithmetic and the harmonic means. We emphasize that here we are speaking about  transformations on the operator structure $\Bpp$ in general and not only about scalar valued functions and the continuous function calculus as transformations which appear in \eqref{E:veg2}.

More precisely, we have the following statement.

\begin{corollary}\label{C:10}
Assume that $\sigma$ is a symmetric Kubo-Ando mean and there is an injective transformation $\phi$ from $\Bpp$ into a linear space (or more generally, into any uniquely 2-divisible commutative semigroup) which satisfies
\begin{equation}\label{E:V1}
\phi(A\sigma B)=\frac{\phi(A)+\phi(B)}{2}, \quad A,B\in \Bpp.
\end{equation}
Then $\sigma$ is either the arithmetic or the harmonic mean.
\end{corollary}

\begin{proof}
For any $A,B,C,D\in \Bpp$, one can compute
\[
\begin{gathered}
4\phi((A\sigma B) \sigma (C\sigma D))=\phi(A)+\phi(B)+\phi(C)+\phi(D)\\=
\phi(A)+\phi(C)+\phi(B)+\phi(D)=
4\phi((A\sigma C)\sigma (B\sigma D)).
\end{gathered}
\]
It follows that $\sigma$ satisfies \eqref{E:71}, hence Theorem \ref{T:H1} applies and we obtain the desired conclusion.
\end{proof}

We remark that 
Proposition 7 in \cite{ML17e} says that if on the positive definite cone of a $C^*$-algebra an injective transformation $\phi$ as in \eqref{E:V1} exists for the geometric mean in the place of $\sigma$, then the algebra is necessarily commutative. Now, the natural question arises if the statement in Corollary \ref{C:10} holds for any non-commutative $C^*$-algebra in the place of $\Bh$.

We can go further and close the paper with
the general question that how the results presented above survive in the setting of general $C^*$-algebras.
As another example, is the statement in Proposition \ref{P:nonex} valid in any non-commutative $C^*$-algebra? In other words, does the existence of a map $M$ satisfying the conditions formulated there on the positive definite cone of a $C^*$-algebra forces the underlying algebra to be commutative?
Concerning the other results: are the assertions in Theorems \ref{T:QA}, \ref{T:assoc}, \ref{T:H1} and Corollaries \ref{C:QA}, \ref{C:10} valid in any (or in some general large classes of) non-commutative $C^*$-algebras?
To point out that these are certainly nontrivial problems, we recall that our main tools to obtain the above mentioned  results in the paper were: the structure theorems of order automorphisms of the positive definite cone and the space of self-adjoint operators in $\Bh$, and the knowledge on strength functions including the formula \eqref{E:36}. No one of those tools has direct extensions for general $C^*$-algebras and this makes our problems hopefully challenging.

\bibliographystyle{amsplain}

\end{document}